\documentclass[12pt]{amsart}

\usepackage{latexsym, amssymb, amsmath,ulem,soul,esint}
\usepackage{tikz}

\usepackage{amsfonts, graphicx}
\usepackage{graphicx,color}
\newcommand{\be}{\begin{equation}}
\newcommand{\ee}{\end{equation}}
\newcommand{\beq}{\begin{eqnarray}}
\newcommand{\eeq}{\end{eqnarray}}

\newtheorem{thm}{Theorem}[section]

\newtheorem{lma}{Lemma}[section]
\newtheorem{prop}{Proposition}[section]
\newtheorem{cor}{Corollary}[section]
\newtheorem{defn}{Definition}[section]

\usepackage{mathrsfs} 

\usepackage{wasysym,stmaryrd}

\theoremstyle{remark}
\newtheorem{rem}{Remark}[section]
\numberwithin{equation}{section}

\def\be{\begin{equation}}
\def\ee{\end{equation}}
\def\bee{\begin{equation*}}
\def\eee{\end{equation*}}

\def\lf{\left}
\def\ri{\right}

\def\K{K\"ahler }

\def\KR{K\"ahler-Ricci }
\def\Ric{\text{\rm Ric}}
\def\Rm{\text{\rm Rm}}

\def\p{\partial}

\def\rheat{\lf(\frac{\p}{\p t}-\Delta_{g(t)}\ri)}

\def\e{\varepsilon}

\def\a{{\alpha}}
\def\b{{\beta}}

\def\C{\mathbb{C}}

\begin{document}

\title[]
{Ricci flow under Kato-type curvature lower bound}

\author{Man-Chun Lee}
\address[Man-Chun Lee]{Department of Mathematics, The Chinese University of Hong Kong, Shatin, N.T., Hong Kong
}
\email{mclee@math.cuhk.edu.hk}

\renewcommand{\subjclassname}{
  \textup{2010} Mathematics Subject Classification}

\date{\today}

\begin{abstract}
In this work, we extend the existence theory of non-collapsed Ricci flows from point-wise curvature lower bound to Kato-type curvature lower bound.  As an application, we prove that compact three dimensional non-collapsed strong Kato limit space is homeomorphic to a smooth manifold.  Moreover, similar result also holds in higher dimension under stronger curvature condition. We also use the Ricci flow smoothing to study stability problem in scalar curvature geometry.
\end{abstract}


\maketitle

\markboth{Man-Chun Lee}{}

\section{Introduction}

There is a long history on the study of compactness of sets of manifolds with uniform lower bound of Ricci curvature.  For instance, the Gromov compactness Theorem \cite{Gromov1999} states that sequence of compact manifolds with uniform Ricci lower bound and diameter upper bound admits a convergent subsequence in the Gromov-Hausdorff topology. Since then,  there have been  many studies toward understanding the structure of the metric space arising as limits of smooth manifolds, see \cite{CheegerColding1997,CheegerColding2000,CheegerColding2000-2,
CheegerNaber2015,CheegerJiangNaber2021,Colding1997} and the reference therein. The analytic properties such as comparison geometry and heat kernel estimates on manifolds with Ricci lower bound play an important role there.

On the other hand, there are however many interesting scenarios in which uniform Ricci lower bound is missing, especially in the study of Ricci flow. It is then natural and important to consider the case when a uniform Ricci lower bound is further relaxed to bound in a weaker sense.  In \cite{PetersenWei1997},  Petersen and Wei generalized the classical fundamental Laplacian and volume comparison for uniform Ricci curvature lower bounds to  smallness assumptions of $L^p$ norm of $\Ric_-$ for $p>n/2$  where $\Ric_-(x)=\inf\{ \a\geq 0: \Ric(g(x))+\a g(x)\geq 0\}$.  This opens the door to understanding the structure of  limit spaces under small $||\Ric_-||_p$ assumption for $p>n/2$, see also \cite{DaiWeiZhang2018,MaWang2021}.  

More recently,  it was further generalized by Carron \cite{Carron2019} and Rose \cite{Rose2019} that in the compact case, a Dynkin-type lower bound of Ricci curvature also suffices to obtain the Li-Yau estimates for the heat kernel, building on the idea of Zhang-Zhu \cite{ZhangZhu2018}.  For a complete Riemannian manifold  $(M^n,g_0)$ of dimension $n\geq 2$,   we denote $$\kappa_{t}(M^n,g_0)=\sup_{x\in M} \int^t_0 \int_M H_{g_0}(x,y,s)(\Ric)_-(y)\, d\mathrm{vol}_{g_0}(y) ds$$
where $H_{g_0}$ is the heat kernel, that is $H(\cdot,\cdot,t)$ is the kernel of the operator $e^{-t\Delta_{g_0}}$ for any $t>0$.

\begin{defn}
Let $\{(M_i,g_i)\}_{i=1}^\infty$ be a sequence of compact manifolds\footnote{
We might analogously define the notions on complete non-compact case.  }. We say that $\{(M_i,g_i)\}_{i=1}^\infty$ satisfies
\begin{itemize}
\item a uniform Dynkin bound on $\Ric_-$ if there exists $T>0$ such that for all $i\in \mathbb{N}$, 
$$\kappa_T(M_i,g_i)\leq \frac1{16n};$$
\item  a uniform Kato bound  on $\Ric_-$ if there exist a non-decreasing function  $f:(0,T]\to (0,+\infty)$, $T>0$ such that $f(t)\to 0$ as $t\to 0^+$ and for all $t\in (0,T]$ and $i\in \mathbb{N}$,
$$\kappa_t(M_i,g_i)\leq f(t);$$
\item a strong uniform Kato bound  on $\Ric_-$ if there exist a non-decreasing function  $f:(0,T]\to (0,+\infty)$, $T,\Lambda>0$ such that $f(t)\to 0$ as $t\to 0^+$ and for all $t\in (0,T]$ and $i\in \mathbb{N}$,
$$\kappa_t(M_i,g_i)\leq f(t);\quad \int^T_0\frac{\sqrt{f(s)}}{s} ds\leq \Lambda.$$
\end{itemize}
\end{defn}
When $\Ric\geq -K$,  it satisfies the uniform Dynkin for $T=(16nK)^{-1}$.  In the compact case, it was also shown that a strong uniform Kato bound  can be achieved for some small $T$ under suitable $L^p$ bound of $\Ric_-$  \cite{RoseStollmann2017} or suitable Morrey bound \cite{CarronRose2021}.  Perhaps more importantly, it was proved by Carron \cite{Carron2019} that the set of compact manifolds satisfying a uniform Dynkin bound on $\Ric_-$  is pre-compact in the Gromov-Hausdorff topology. When it is strengthened to strong uniform Kato sense,  the structure theory of the corresponding volume non-collapsed Gromov-Hausdorff limit has been developed by Carron-Mondello-Tewodrose \cite{CarronMondelloTewodrose2021}.  This generalizes the earlier works of Cheeger-Colding on Ricci limit spaces. We refer readers to \cite{CarronMondelloTewodrose2022,CarronMondelloTewodrose2023,Rose2021,
 RoseWei2022} for more interesting and important developments.

In this work, we are interested in the regularity of the Gromov-Hausdorff limits of volume non-collapsed compact manifolds (without boundary) under the strong Kato lower bound on $1$-isotropic curvature.  To clarify the notion, denote the space of algebraic curvature tensors on $\mathbb{R}^n$ by $\mathcal{C}_B(\mathbb{R}^n)$.  For any given $R\in \mathcal{C}_B(\mathbb{R}^n)$, we extend it complex linearly to $\mathbb{C}^n$.  We say that $R\in\mathrm{C}_{\mathrm{PIC2}}$ if for each two complex dimensional subspace $\Sigma$ of $\mathbb{C}^n$ and orthonormal basis $v,w\in \mathbb{C}^n$ of $\Sigma$, we have  $R(v,w,\bar v,\bar w)\geq 0$.  If one instead asks for non-negativity of complex sectional curvature only for PIC1 sections, defined to be those $\Sigma$ that contain some non-zero vector $v$ whose conjugate $\bar v$ is orthogonal to the entire section $\Sigma$. The algebraic curvature tensors $R$ with non-negative complex sectional curvature for each such restricted $\Sigma$ form a cone we denote by $\mathrm{C}_{\mathrm{PIC1}}$.  When $n=3$, it is known that $\mathrm{C}_{\mathrm{PIC1}}$ is also the cone of curvature tensors $R$ with $\Ric(R)\geq 0$.  

For $n\geq 4$,  one can equivalently describe $\mathrm{C}_{\mathrm{PIC1}}$ as follows (see \cite{BrendleBook}): $\mathrm{C}_{\mathrm{PIC1}}$ is the cone consisting of curvature tensors $R$ such that for any orthonormal four-frame $\{e_i\}_{i=1}^4\subset \mathbb{R}^n$ and $\lambda\in [0,1]$, we have
$$R_{1313}+\lambda^2 R_{1414}+R_{2323}+\lambda^2R_{2424}-2\lambda R_{1234}\geq 0.$$
We are interested in studying manifolds where the curvature is bounded from below with respect to $\mathrm{C}_{\mathrm{PIC1}}$.  For a Riemannian manifold $(M^n,g)$ with $n\geq 3$ and $x\in M$, we define
$$(\mathrm{IC}_{1})_-(x)=\inf\{ \a\geq 0: \mathrm{Rm}_{g(x)}+\a \mathrm{I}\in\mathrm{C}_{\mathrm{PIC1}}\}$$
where $\mathrm{I}$ is the curvature tensor defined by $I_{ijkl}=\delta_{ik}\delta_{jl}-\delta_{il}\delta_{jk}$. This is the negative part of the lowest eigenvalue of $\mathrm{Rm}$ with respect to $\mathrm{C}_{\mathrm{PIC1}}$. This generalizes the notion of $\Ric_-$ in the sense that  when $n=3$,  $\Ric_-=(n-1)\cdot (\mathrm{IC}_{1})_-$ while when $n\geq 4$,  we have $\Ric_-\leq (n-1)\cdot (\mathrm{IC}_{1})_-$.

We follow the spirit of $\kappa_t$ and define, for all $t>0$,
$$\kappa_{t,\mathrm{IC1}}(M^n,g_0)=\sup_{x\in M} \int^t_0 \int_M H_{g_0}(x,y,s)(\mathrm{IC}_{1})_-(y)\, d\mathrm{vol}_{g_0}(y) ds.$$

\begin{defn}Suppose $T>0$ and $f:(0,T]\to (0,+\infty)$ is a non-decreasing function such that 
\begin{equation}\label{K-condition}
\lim_{t\to 0}f(t)=0 ,\quad\text{and}\quad f(T)\leq \frac1{16n(n-1)}.
\end{equation}
We say that $(M^n,g)\in \mathcal{K}_{\mathrm{IC1}}(n,f)$ if for all $t\in (0,T]$ we have
$$\kappa_t(M^n,g_0)\leq f(t).$$
If in addition,  $ \mathrm{Vol}_g\left(B_g(x_0,\sqrt{T})\right)\geq v T^{n/2}$ for some $x_0\in M,v>0$, then we say that $(M^n,g,x_0)\in  \mathcal{K}_{\mathrm{IC1}}(n,f,v)$.
\end{defn}

We will mainly focus on the case when $M$ is compact and the Kato-bound is uniformly strong.  Given a sequence of compact manifolds $\{(M_i,g_i)\}_{i=1}^\infty$, we say that $\{(M_i,g_i)\}_{i=1}^\infty$ satisfies a strong uniform $\mathrm{IC1}$-Kato if $(M_i,g_i)\in \mathcal{K}_{\mathrm{IC1}}(n,f)$ where $f:(0,T]\to (0,+\infty)$  is a non-decreasing function, $T,\Lambda>0$ such that $f(t)\to 0$ as $t\to 0^+$ and satisfies
\begin{equation}\label{intro:SK-condition}
\int^T_0 \frac{\sqrt{f(s)}}{s}\, ds \leq\Lambda.
\end{equation}
When $n=3$, it coincides with the concept of strong Kato bound of $\Ric_-$ considered in \cite{CarronMondelloTewodrose2021}, modulus scaling\footnote{there will be a  normalization constant $n-1$ since the positivity of $\mathrm{IC}_1$ is defined using language of cone.}.


%
%
%

In contrast with the work \cite{CarronMondelloTewodrose2021,CarronMondelloTewodrose2022}, we intend to strengthen the curvature lower bound from Ricci curvature to $1$-isotropic curvature (which two notions coincide when $n=3$).  This is largely motivated by the works of Bamler-Cabezas-Wilking \cite{BamlerCabezasWilking2019},  Simon-Topping \cite{SimonTopping2021}, Lai \cite{Lai2019} and Lee-Tam \cite{LeeTam2021} on obtaining the regularity of the Gromov-Hausdorff limit using Ricci flow.

The following is the main technical result which says that such manifolds can be regularized  by the Ricci flow uniformly.  This generalizes the earlier work of Bamler-Cabezas-Wilking \cite{BamlerCabezasWilking2019} on the theory of Ricci flows with curvature bounded from below in point-wise sense initially.
\begin{thm}\label{thm:RF-compact}
Suppose $(M^n,g_0,x_0),n\geq 3$ is a pointed compact manifold which is inside $\mathcal{K}_{\mathrm{IC1}}(n,f,v)$ for some non-decreasing function $f$, $T$ and $\Lambda$ satisfying \eqref{intro:SK-condition} and $\mathrm{diam}(M,g_0)\leq D$ for some $D>0$, then there exist $S,\a,\tilde v$ depending only on $n,f,v,D,T,\Lambda$ and a solution $g(t),t\in [0,S]$ to the Ricci flow such that for all $(x,t)\in M\times (0,S]$,
\begin{enumerate}
\item[(a)] $|\Rm(x,t)|\leq \a t^{-1}$;
\item [(b)] $\mathrm{Vol}_{g(t)}\left(B_{g(t)}(x, \sqrt{T})\right)\geq \tilde v T^{n/2}$;
\item[(c)]$\mathrm{inj}(g(x,t))\geq \sqrt{\a^{-1}t}$;
\item[(d)] $\mathrm{Rm}(g(x,t))+\a (1+t^{-1}f(\a t))\cdot \mathrm{I} \in \mathrm{C}_{\mathrm{PIC1}}$.
\end{enumerate}
\end{thm}
A similar result also holds in the \K case under strong uniform Kato lower bound on bisectional curvature,  see Theorem~\ref{thm:kahler-case}.

\begin{rem}
Using the partial Ricci flow approach in \cite{Lai2019} (see also \cite{SimonTopping2021,LeeTam2021}), the result also holds in the complete non-compact case \textit{without} bounded curvature assumption if suitable \textit{static} heat kernel estimates hold.  In the presence of uniform Ricci lower bound, this usually follows from Li-Yau estimates.  In the compact case,  it has been obtained by Carron \cite{Carron2019} and Rose \cite{Rose2019} under the Kato-type lower bound of Ricci curvature. The comparison geometry in the complete non-compact case seems out of reach at the moment due to the absence of exhaustion function with good control.  We will therefore only focus on the compact case in this work.
\end{rem}

As an application,  we prove that the Gromov-Hausdorff limit of sequence of compact manifolds satisfying the strong uniform $\mathrm{IC1}$-Kato bound and a uniform diameter upper bound is homeomorphic to a smooth manifold. 
\begin{cor}\label{cor:RF-fromSing}Suppose $(M^n_i,g_{i,0})$ is a sequence of compact manifold such that 
\begin{enumerate}
\item[(a)] $(M_i,g_{i,0},x_i)\in \mathcal{K}_{\mathrm{IC1}}(n,f,v)$ for some $x_i\in M_i, v>0$ and a non-decreasing function $f$ satisfying \eqref{intro:SK-condition};
\item[(b)]$\mathrm{diam}(M_i,g_{i,0})\leq D$ for some $D>0$.
\end{enumerate}
Then there exist a smooth compact manifold $M_\infty$ and a distance function $d_\infty$ on $M_\infty$ so that $(M_i,g_{i,0})$ sub-converges to $(M_\infty,d_\infty)$ in the measured Gromov-Hausdorff sense.  Moreover, there exists a Ricci flow $g_\infty(t),t\in (0,S]$ on $M_\infty$ such that $$\lim_{t\to 0} d_{GH}\left( (M_\infty,d_\infty), (M_\infty,g_\infty(t)\right)=0.$$

Furthermore, if in addition $(M_i,g_{i,0},x_i)\in \mathcal{K}_{IC1}(n,f_i,v)$ for some decreasing function $f_i$ in which $f_i(T)\to 0$, then $\mathrm{Rm}(g_\infty(t))\in \mathrm{C}_{\mathrm{PIC1}}$.
\end{cor}

The final part of the result also holds if $\mathrm{C}_{\mathrm{PIC1}}$ is replaced by other curvature cone in Lemma~\ref{lma:curvSatisfyingProp}. Moreover, it particularly infers that if a compact manifold supports a metric with sufficiently small negative lower bound of $\mathrm{IC}_1$ in the strong Kato sense relative to $n,D,v$, then it also admits a smooth metric whose curvature operator lies in $\mathrm{C}_{\mathrm{PIC1}}$ via a contradiction argument.

Another application is to understand the stability of metrics on torus with almost non-negative scalar curvature.  
When $M=\mathbb{T}^n$, the celebrated work of Schoen and Yau \cite{SchoenYau1979,SchoenYau1979-2}, Gromov and Lawson \cite{GromovLawson1980} stated that metrics with non-negative scalar curvature $\mathcal{R}\geq 0$ must be flat. In \cite{Gromov2014}, Gromov asked if a sequence of metrics $g_i$ with $\mathcal{R}(g_i)\geq -i^{-1}$ on $\mathbb{T}^n$ will sub-converge to a flat metric in some appropriate weak sense.  For related works, we refer interested readers to the survey paper of Sormani \cite{Sormani2021} for a comprehensive discussion. In this regard, we show that if in addition the $1$-isotropic curvature is bounded from below in the strong Kato sense, then the sequence will sub-converge to the flat torus in the measured Gromov-Hausdorff sense. 
\begin{cor}\label{cor:LimitinPSC}
Suppose $(M^n_i,g_{i,0})$ is a sequence of compact manifolds such that
\begin{enumerate}
\item[(a)] $(M_i,g_{i,0},x_i)\in \mathcal{K}_{\mathrm{IC1}}(n,f,v)$ for some $x_i\in M_i, v>0$ and a non-decreasing function $f$ satisfying \eqref{SK-condition};
\item[(b)]$\mathrm{diam}(M_i,g_{i,0})\leq D$ for some $D>0$;
\item[(c)] $\sigma(M_i)\leq 0$ (equivalently, $M_i$ does not admit PSC metric);
\item[(d)]$\mathcal{R}(g_i)\geq -i^{-1}$.
\end{enumerate}
Then after passing to subsequent, $(M_i^n,g_{i,0})$ converges to a compact Ricci flat manifold $(M_\infty,g_{\infty})$ in the measured Gromov-Hausdorff sense.  If in addition $b_1(M_i)=n$ for all $i$, then $(M_\infty,g_\infty)$ is a flat torus.
\end{cor}


{\it Acknowledgement}: The author would like to thank C. Rose for his interest and D. Tewodrose for comments on improving the manuscript. The work was partially supported by Hong Kong RGC grant (Early Career Scheme) of Hong Kong No. 24304222, a direct grant of CUHK and a NSFC grant.

\section{the strong kato bound and implications}

In the following, compact manifolds are referred those without boundary. In this section, we will collect some important properties from \cite{CarronMondelloTewodrose2021}  for compact manifolds $(M^n,g)$ satisfying the strong Kato bound.  The properties will play crucial role in obtaining estimates of Ricci flows.

Since $\mathrm{C}_{\mathrm{PIC1}}\subseteq \mathrm{C}_{\Ric}$ for $n\geq 3$, it is clear that for all $t\in (0,+\infty)$, 
$$\kappa_{t}(M^n,g_0)\leq  (n-1)\kappa_{t,\mathrm{IC1}}(M^n,g_0).$$
 
 Assume $T>0$ and $f:(0,T]\to (0,+\infty)$ such that $f$ is non-decreasing and satisfies
\begin{equation}\label{weak-Kato}
\lim_{t\to 0}f(t)=0 ,\quad\text{and}\quad f(T)\leq \frac1{16n}.
\end{equation}
We say that $(M^n,g)\in \mathcal{K}(n,f)$ if for all $t\in (0,T]$ we have $\kappa_t(M^n,g)\leq f(t)$. Similarly, we define $ \mathcal{K}(n,f,v)$ to be the class of pointed compact manifolds $(M,g,x_0)$ such that the volume $ \mathrm{Vol}_g\left(B_g(x_0,\sqrt{T})\right)\geq v T^{n/2}$. The following Proposition says that the heat kernel will behave in a similar way as the Euclidean one under the Kato control. This is originated in \cite{Carron2019}
\begin{prop}{\cite[Proposition 2.6]{CarronMondelloTewodrose2022}}\label{prop:heat-Kato}
There exists $\gamma(n)\geq 1$ such that if $(M^n,g)\in \mathcal{K}(n,f)$, then for all $x,y\in M$ and $t\in (0,T)$,
$$\frac{\gamma^{-1}}{\mathrm{Vol}_g\left( B_g(x,r)\right)}\cdot \exp\left(-\frac{\gamma d^2(x,y)}{t} \right)\leq H_{g_0}(x,y,t)\leq \frac{\gamma}{\mathrm{Vol}_g\left( B_g(x,r)\right)}\cdot \exp\left(-\frac{d^2(x,y)}{5t} \right).$$
\end{prop}
Moreover, in this case, it is known that $(M^n,g)\in \mathcal{K}(n,f)$ is a volume doubling space: for any $x\in M$, $0<s<r\leq \sqrt{T}$,
$$\frac{\mathrm{Vol}_g\left( B_g(x,r)\right)}{\mathrm{Vol}_g\left( B_g(x,s)\right)}\leq C\left( \frac{r}{s}\right)^{e^2n}.$$

If in addition $f$ satisfies a stronger integrability condition: 
\begin{equation}\label{SK-condition}
\int^T_0 \frac{\sqrt{f(s)}}{s}\, ds \leq\Lambda
\end{equation}
for some $\Lambda> 0$. As pointed out in \cite{CarronMondelloTewodrose2021},  $f$ satisfying \eqref{SK-condition} automatically implies that $f$ satisfies \eqref{weak-Kato} by shrinking $T$.   It was first shown by Carron \cite{Carron2019} that under the stronger assumption \eqref{SK-condition}, we can strength the conclusion to Ahlfors $n$-regular. 
\begin{prop}{\cite[Proposition 2.14 \& Corollary 2.15]{CarronMondelloTewodrose2021}}\label{prop:vol-SK}
Suppose $(M^n,g)\in \mathcal{K}(n,f)$ and $f$ satisfies \eqref{SK-condition} for some $\Lambda>0$. Then there exists $C_n>0$ such that for all $x\in M$ and $0<s<r\leq \sqrt{T}$,
$$\mathrm{Vol}_g\left( B_g(x,r)\right)\leq C_n^{\Lambda+1} r^n\quad\text{and}\quad \frac{\mathrm{Vol}_g\left( B_g(x,r)\right)}{\mathrm{Vol}_g\left( B_g(x,s)\right)}\leq C_n^{\Lambda+1}\left( \frac{r}{s}\right)^{n}.$$
Moreover,  for any $x,y\in M$ and $0\leq r\leq \sqrt{T}$,
$$\frac{\mathrm{Vol}_g\left( B_g(x,\sqrt{T})\right)}{T^{n/2}}\leq C_n^{(\Lambda+1)T^{-1/2} d(x,y)}\cdot  \frac{\mathrm{Vol}_g\left( B_g(y,r)\right)}{r^n}.$$ 
\end{prop}

We are primarily interested in the non-collapsed case in this work.  
We end this section by observing that the non-collapsed strong Kato assumption can be re-phased as an estimate in an integral form.

%

\begin{lma}\label{lma:SKT-Morry}
Suppose $(M^n,g,x_0)\in \mathcal{K}_{\mathrm{IC1}}(n,f)$ for some $f$ satisfying \eqref{SK-condition}, then there exists $C_1(n,\Lambda,T)>0$ such that for all $(x,t)\in M\times (0,T]$,
$$\int_M \frac1{t^{n/2}} \exp\left(-\frac{C_1 d_{g_0}^2(x,y)}{t} \right)(\mathrm{IC}_{1})_-(y) \; d\mathrm{vol}_{g_0}(y)\leq \frac{C_1 f(t)}{t}.$$
\end{lma}
\begin{proof}
Since $\kappa_t(M,g)\leq (n-1)\kappa_{t,\mathrm{IC1}}$, we have $\mathcal{K}_{\mathrm{IC1}}(n,f)\subseteq \mathcal{K}(n,f)$ for a possibly larger $\Lambda$. By proposition~\ref{prop:heat-Kato}, Proposition~\ref{prop:vol-SK} and the volume non-collapsing, there exist $C_0(n,\Lambda,T),\gamma(n)>0$ such that for all $(x,t)\in M\times (0,T]$,
\begin{equation}
\begin{split}
&\quad  \int^t_0 \int_M H(x,y,s)(\mathrm{IC}_{1})_-(y)\, d\mathrm{vol}_{g_0}(y) ds\\
&\geq \int^t_0 \int_M \frac{\gamma}{V_g(x,\sqrt{s})} \exp\left(-\frac{\gamma d_{g_0}^2(x,y)}{s} \right)(\mathrm{IC}_{1})_-(y) \; d\mathrm{vol}_{g_0}(y) ds\\
&\geq \frac\gamma{C_n^{\Lambda+1}}\int^t_0 \int_M \frac{ 1}{s^{n/2}} \exp\left(-\frac{\gamma d_{g_0}^2(x,y)}{s} \right)(\mathrm{IC}_{1})_-(y) \; d\mathrm{vol}_{g_0}(y) ds.
\end{split}
\end{equation}

Combines this with the assumption of Kato-type bound, we have for all $(x,t)\in M\times (0,T]$,
\begin{equation}
\begin{split}
f(t)&\geq \frac\gamma{C_n^{\Lambda+1}}\int^t_{t/2} \int_M \frac1{s^{n/2}} \exp\left(-\frac{\gamma d_{g_0}^2(x,y)}{s} \right)(\mathrm{IC}_{1})_-(y) \; d\mathrm{vol}_{g_0}(y) ds\\
&\geq \frac{t\gamma}{2 C_n^{\Lambda+1}}\int_M \frac1{t^{n/2}} \exp\left(-\frac{2\gamma  d^2_{g_0}(x,y)}{t} \right)(\mathrm{IC}_{1})_-(y)\; d\mathrm{vol}_{g_0}(y).
\end{split}
\end{equation}
This completes the proof by re-arranging the inequality.
\end{proof}

\section{Ricci flow smoothing and estimates}

The idea of this work is to mollify the metric $g_0$ using the Ricci flow. The Ricci flow is a one parameter family of metrics $g(t)$ satisfying 
\begin{equation}
\partial_t g(t)=-2\Ric(g(t)),\quad\quad g(0)=g_0.
\end{equation}
In the coming few sections, we will focus on obtaining estimates of Ricci flows. 

If $(M^n,g_0)$ is a compact manifold, it was shown by Hamilton \cite{Hamilton1982} that the Ricci flow admits a short-time solution $g(t)$.  It is by-now known to be extremely powerful in regularizing metrics even in case without bounded curvature, for instances see the works \cite{HuangWang2022,Wang2020,Lai2019, SimonTopping2021,HochardPhd} and the reference therein.  Although we are primarily interested in the compact case, we will state our result in this section using a localized form so that it also applies to the non-compact case. 

It is a general philosophy that the Ricci flow will tend to be positively curved as it evolves, for instances see the list in Lemma~\ref{lma:curvSatisfyingProp} and the recent work of Brendle \cite{Brendle2019}.  The following Proposition illustrates its relation to curvature estimates in the non-collapsed case. This is a slight generalization of the curvature estimates in \cite{Lai2019,LeeTam2021,SimonTopping2023}.
\begin{prop}\label{Prop:CurvEsti}
For any $n\geq 3,v_0,L>0$,  there exist $\tilde S(n,v_0,L),a(n,v_0,L)>0$ such that the following holds. Suppose $(M^n,g(t)), t\in [0,S]$ is a smooth solution to the Ricci flow (not necessarily complete) such that for some $x_0\in M$, we have
\begin{enumerate}
\item[(a)]$B_{g(t)}(x_0,1)\Subset M$ for all $t\in [0,S]$;
\item[(b)] for all $x\in B_{g(t)}(x_0,\frac34),t\in [0,S]$ and $r\in (0,\frac14\sqrt{L^{-1}t}\wedge \frac14]$,
$$\mathrm{Vol}_{g(t)}\left(B_{g(t)}(x,r) \right)\geq v_0 r^n.$$
\item[(c)] for all $x\in B_{g(t)}(x_0,1)$ and $t\in (0,S]$, we have either
\begin{enumerate}
\item[(i)] $n=3$ or;
\item[(ii)] $\mathrm{IC}_1(g(t))\geq -Lt^{-1}$ or;
\item[(iii)] $\mathrm{OB}(g(t))\footnote{OB denotes the orthogonal bisectional curvature in the K\"ahler case}\geq -Lt^{-1}$ when $n$ is even and  $g(t)$ is K\"ahler.
\end{enumerate}
\end{enumerate}  
Then for $t\in (0,\tilde S\wedge S]$ and $x\in B_{g(t)}(x_0,\frac12)$, we have 
$$|\Rm(x,t)|\leq at^{-1}.$$
\end{prop}
\begin{proof}
The proof of the curvature estimate uses the Perelman-inspired point-picking argument from \cite[Lemma 2.1]{SimonTopping2023}.  We illustrate the case of (ii) and will point out the modification needed for (i) and (iii).

Suppose the conclusion is false for some $n,v_0,L>0$. Then for any $a_k\to +\infty$, we can find a sequence of manifold $M^n_k$, Ricci flows $g_k(t),t\in [0,S_k]$ and $p_k\in M_k$ satisfying the hypotheses but so the curvature estimate fails with $a=a_k$ in an arbitrarily short time.  We might assume $a_kS_k\to 0$.  Using the smoothness of the flow, we can choose $t_k\in (0,S_k]$ such that 
\begin{enumerate}
\item[(i)] $B_{g_k(t)}(x_k,1)\Subset M_k$ for $t\in [0,t_k]$;
\item[(ii)] $ \mathrm{Vol}_{g_k(t)}\left(B_{g_k(t)}(x,r) \right)\geq v_0r^n$  for $t\in [0,t_k]$, $x\in B_{g_k(t)}(x_k,\frac34)$ and $r\in (0,\frac14\sqrt{L^{-1}t}]$;
\item[(iii)] $\mathrm{IC}_1(g_k(t))\geq -Lt^{-1}$ on $B_{g_k(t)}(x_k,1)$ for $t\in [0,t_k]$;
\item[(iv)] $|\Rm_{g_k}(x,t)|<a_k t^{-1}$ for $t\in [0,t_k]$ and $x\in B_{g_k(t)}(x_k,\frac12)$;
\item[(v)]  $|\Rm_{g_k}(z_k,t_k)|=a_k t_k^{-1}$ for some $z_k\in \overline{B_{g_k(t_k)}(x_k,\frac12)}$.
\end{enumerate}

By (v) and the fact that $a_kt_k\to 0$,  \cite[Lemma 5.1]{SimonTopping2023} implies that for sufficiently large $k\in\mathbb{N}$, we can find $\b(n)>0$, $\tilde t_k\in (0,t_k]$ and $\tilde x_k\in B_{g_k(\tilde t_k)}(x_k,\frac34 -\frac12 \b \sqrt{a_k\tilde t_k})$ such that 
\begin{equation}\label{eqn:curEsti}
|\Rm_{g_k}(x,t)|\leq 4|\Rm_{g_k}(\tilde x_k,\tilde t_k)|=4Q_k
\end{equation}
whenever $d_{g_k(\tilde t_k)}(x,\tilde x_k)<\frac18 \b a_k Q_k^{-1/2}$ and $\tilde t_k-\frac18 a_kQ_k^{-1}\leq t\leq \tilde t_k $ where $\tilde t_kQ_k \geq a_k\to +\infty$. We note here that from the proof of  \cite[Lemma 5.1]{SimonTopping2023}, for $(x,t)\in \frac18 \b a_k Q_k^{-1/2}\times [\tilde t_k-\frac18 a_kQ_k^{-1}, \tilde t_k]$, we have $d_{g_k(t)}(x,x_k)<1$ so that the conditions applied.


We now consider the parabolic rescaled Ricci flows $\tilde g_k(t)=Q_k g(t_k+Q_k^{-1}t),t\in [-\frac18 a_k,0]$ so that 
\begin{enumerate}
\item[(a)] $ B_{\tilde g_k(0)}(\tilde x_k,\frac18 \b a_k)\Subset M_k$;
\item[(b)] $|\Rm_{\tilde g_k(t)}|\leq 4$ on $B_{\tilde g_k(0)}(\tilde x_k,\frac18 \b a_k)\times [-\frac18a_k,0]$;
\item[(c)]  $|\Rm_{\tilde g_k}(\tilde x_k,0)|=1$;
\item[(d)] $ \mathrm{Vol}_{\tilde g_k(0)}\left(B_{g_k(t)}(\tilde x_k,r) \right)\geq v_0r^n$ for $r\in (0,\frac14 \sqrt{L^{-1}Q_kt_k}]$;
\item[(e)] $ \displaystyle\mathrm{IC}_1(\tilde g_k(t))\geq -\frac{L}{t+Q_k\tilde t_k}$ on $ B_{\tilde g_k(0)}(\tilde x_k,\frac18 a_k)\times  [-\frac18a_k,0]$.
\end{enumerate}

By (b), (d) and the result of Cheeger, Gromov and Taylor \cite{CheegerGromovTaylor1982}, the injectivity radius of $\tilde g_k(0)$ at $\tilde x_k$ is bounded from below uniformly.  Together with the curvature estimates from (b), we may apply Hamilton’s compactness theorem \cite{Hamilton1995} to conclude that, $(M_k,\tilde g_k(t),\tilde x_k)$ sub-converges in the Cheeger-Gromov sense to $(M_\infty,\tilde g_\infty(t),\tilde x_\infty)$ which is a complete non-flat ancient solution to the Ricci flow with bounded curvature.  By (e),  $\mathrm{Rm}(\tilde g_\infty(t))\in \mathrm{C}_{\mathrm{PIC1}}$ for all $t\leq 0$. Moreover,  (d) implies that it is of Euclidean volume growth. But this contradicts with \cite[Lemma 4.2]{BamlerCabezasWilking2019}. This completes the proof. 

In case of (iii), each $g_k(t)$ is in addition K\"ahler and the conclusion (e) is replaced by $\mathrm{OB}(\tilde g_k(t))\geq -\frac{L}{t+Q_k\tilde t_k}$ instead so that the limiting Ricci flow $\tilde g_\infty(t)$ is K\"ahler, non-flat with Euclidean volume growth and has non-negative bounded holomorphic orthogonal bisectional curvature.  This contradicts with  \cite[Proposition 6.1]{LiNi2020}.

In case of (i),  the Ricci curvature of $\tilde g_k(t)$ is not almost non-negative but we might instead apply Hamilton-Ivey estimates (for instances see \cite[Proposition 9.8]{ChowBookI}) on the limiting Ricci flow $\tilde g_\infty(t)$ to conclude that $\tilde g_\infty(t)$ has non-negative sectional curvature and Euclidean volume growth. The non-flatness contradicts with Perelman's Theorem \cite[Section 11]{Perelman2002}. 
\end{proof}

Thanks to Proposition~\ref{Prop:CurvEsti},  obtaining curvature estimates is equivalent to establishing the persistence of volume non-collapsing and curvature lower bound.  To obtain a the curvature lower bound along the flow, we will need a version of the local maximum principle which allow the initial data to be merely unbounded.  This is a modified version of the local maximum principle in \cite{HochardPhd,LeeTam2022}. 
{
\begin{prop}\label{prop:MP}
Suppose $g(t)$ is a smooth solution to the Ricci flow on $M\times [0,S]$ (not necessarily) and there exist $\a,r>0$ so that all $t\in [0,S]$, 
\begin{enumerate}
\item[(a)] $|\Rm(g(t))|\leq \a t^{-1}$;
\item[(b)] $\mathrm{inj}(g(t))\geq \sqrt{\a^{-1}t}$;
\item[(c)] either 
\begin{enumerate}
\item[(i)] $M$ is compact or;
\item[(ii)]$B_{g_0}(x_0,2r)\Subset B_{g(t)}(x_0,4r)\subseteq B_{g(t)}(x_0,16r)\Subset M$ for some $x_0\in M$.
\end{enumerate}
\end{enumerate}
 If $\varphi$ is a continuous non-negative function on $M\times [0,S]$ with $\varphi(0)=\varphi_0$ such that
\begin{enumerate}
\item[(i)] $\varphi(t)\leq \a t^{-1}$ on $M\times (0,S]$;
\item[(ii)]  $\varphi$ satisfies 
$$\rheat \varphi\leq \mathcal{R}\varphi+K\varphi^2$$
in the sense of barrier for some constant $K\geq 0$;
\item[(iii)]  There exist $\Lambda,C_1>0$ and a non-decreasing $f:(0,r^2]\to (0,+\infty)$ satisfying $\int^{r^2}_0 s^{-1}f(s) ds \leq \Lambda$ so that for all $x\in B_{g_0}(x_0,r)$ and $t\in (0,r^2]$,
$$\int_{B_{g_0}(x_0,2r)}\frac1{t^{n/2}} \exp\left(-\frac{C_1}{t}d_{g_0}^2(x,y) \right)\varphi_0(y) \; d\mathrm{vol}_{y,g_0} \leq \frac{C_1f(t)}{t}.$$
\end{enumerate}
Then there exist $\tilde\Lambda(n,\a,f,\Lambda,C_1),\tilde S(K,n,\a,f,\Lambda,C_1)>0$ so that for all $t\in (0,\tilde S r^2\wedge S]$,
$$ \varphi(x_0,t)\leq \tilde\Lambda  \left( \frac{f(\tilde\Lambda  t)}{t}+r^{-2}\right).$$
\end{prop}
\begin{proof}
By considering $\tilde g(t)=r^{-2}g(r^2t),\tilde f(t)=f(r^2t)$ and $\tilde\varphi(t)=r^2\varphi(r^2t)$, we might assume $r=1$. 
For each $x\in B_{g_0}(x_0,1)$, we define
$$\rho(x)=\sup\left\{r\in (0,1): B_{g_0}(x,r)\Subset B_{g_0}(x_0,1)\right \}$$

Let $\tilde\Lambda,S$ be constants to be determined as well as a differentiable function $\mathfrak{F}:(0,+\infty)\to (0,+\infty)$ satisfying $\lim_{t\to 0^+}\mathfrak{F}(t)=0$ and $\lim_{t\to 0^+}\mathfrak{F}'(t)=+\infty$. We will specify the choices in the proof.

We want to show that if $\tilde \Lambda$ is chosen to be large enough, $\tilde S$ is sufficiently small and $\mathfrak{F}(t)$ is chosen appropriately, then for all $(x,t)\in B_{g_0}(x_0,1)\times (0,S\wedge \tilde S]$, 
$$ \varphi(x,t)< \tilde\Lambda \left(\mathfrak{F}'(t)+\frac{1}{\rho^2(x)} \right).$$
Since $\mathfrak{F}'$ is unbounded as $t\to 0^+$, there exists $t_1\in (0,S\wedge \tilde S]$ such that the conclusion holds on $(0,t_1)$. If $t_1=S\wedge \tilde S$, then we are done. Otherwise, there exists $x_1\in B_{g_0}(x_0,1)$ so that $\varphi(x_1,t_1)=\tilde \Lambda \left(\mathfrak{F}'(t_1)+{\rho^{-2}(x_1)} \right)$.   We denote $\rho_1=\rho(x_1)$ so that for all $z\in B_{g_0}(x_1,\frac12\rho_1)$, $\rho(z)\geq \frac12 \rho_1$. In particular,  for all $(z,t)\in B_{g_0}(x_1,\frac12\rho_1)\times (0,t_1]$,
$$\varphi(z,t)\leq \tilde \Lambda \left(\mathfrak{F}'(t)+\frac{4}{\rho^2(x)} \right).$$
Thus, if we define 
$$\tilde\varphi(z,t)=e^{-4K\tilde \Lambda (\mathfrak{F}(t)+t\rho_1^{-2})} \varphi(z,t),$$
then it satisfies $\rheat \tilde\varphi\leq \mathcal{R} \tilde\varphi$ on $B_{g_0}(x_1,\frac12\rho_1)\times (0,t_1]$.

By shrinking $\tilde S$, we might assume $t_1\leq 1$. We might then apply \cite[Proposition 4.1]{LeeTam2022} to the rescaled Ricci flow $\tilde g(t)=t_1^{-1}g(t_1t),t\in [0,1]$ with $\Omega=B_{g_0}(x_0,2)=B_{\tilde g(0)}(x_0,2t_1^{-1/2})$. The same also holds if $M$ is compact, see \cite[Remark 4.1]{LeeTam2022}. By rescaling back, we obtain a Dirichlet heat kernel $G(x,t;y,s)$ on $\Omega\times\Omega\times [0,t_1]\times [0,t_1]$ to the operator $\partial_t-\Delta_{g(t)}-\mathcal{R}$ which satisfies
\begin{equation}
\label{heat-Kernel}G(x,t;y,s)\leq \frac{C_2}{(t-s)^{n/2}} \exp\left(-\frac{d_{g(s)}^2(x,y)}{C_2(t-s)} \right)
\end{equation}
for some $C_2(n,\a)>0$.  Consider the function
$$u(x,t)=\int_{\Omega} G(x,t;y,0)\varphi_0(y)\, d\mathrm{vol}_{g_0}(y)$$ 
on $\Omega\times [0,t_1]$ which satisfies $\rheat u=\mathcal{R} u$. If $\tilde S\leq \min\{1,C_2^{-1}C_1\}$, then the heat kernel estimate \eqref{heat-Kernel} shows that  for $x\in B_{g_0}(x_0,1)$ and $t\in [0,t_1]$,
\begin{equation}\label{HeatKern-bdd}
\begin{split}
u(x,t)&= \int_{\Omega}  G(x,t;y,0)\,\varphi_0(y) \, d\mathrm{vol}_{g_0}(y)\\
&\leq  \int_{B_{g_0}(x_0,2)} \varphi_0(y) \cdot \frac{C_2}{t^{n/2}}  \exp \left(-\frac{d_{g_0}^2(x,y)}{C_2t} \right)\,d\mathrm{vol}_{g_0}(y)\\
&\leq C_2C_1^{-1}\frac{f(C_2C_1^{-1}t)}{t}.
\end{split}
\end{equation}
 Thanks to the assumption of $f$, we have $f(t)\leq \Lambda (\log (t^{-1}))^{-1}$. Hence by further shrinking $\tilde S\leq \min\{C_2^{-1}C_1,1,\exp(-\Lambda C_2C_1^{-1})\}$, we have $u\leq t^{-1}$ for $t\in (0,t_1]$. By \cite[Theorem 1.1]{LeeTam2022}, we deduce that there exists $T_1(n,\a)>0$ such that if $t_1\leq \tilde S\wedge S\wedge T_1 \rho_1^2$, then
\begin{equation}
\tilde \varphi(x_1,t_1)\leq u(x_1,t_1)+ 16(t_1 \rho_1^{-2}) \cdot \rho_1^{-2}
\end{equation}

We choose
$$\mathfrak{F}(t)=C_2C_1^{-1}\int^{C_2C_1^{-1}t}_0\left( \frac{f(s)}{s}+\e s^{-1/2}\right)ds$$
for $\e>0$ so that $\mathfrak{F}'(t)\to +\infty$ as $t\to 0^+$. 
Moreover, $\lim_{t\to 0^+}\mathfrak{F}(t)=0$ follows from absolutely continuity of integrable function and assumption on $f$.

On the other hand, using assumption (i), 
$$\frac{\tilde\Lambda}{\rho_1^2} < \tilde\Lambda \left(\mathfrak{F}'(t_1)+\frac{1}{\rho^2_1} \right)\leq \frac{\a}{t_1}$$
giving $t_1\leq \a \tilde\Lambda^{-1}\rho_1^2$ and hence
\begin{equation}\label{contra-point}
\begin{split}
&\quad 4\tilde\Lambda e^{-4K\tilde \Lambda \mathfrak{F}(t_1)-4K\a}\cdot \left(\mathfrak{F}'(t_1)+\rho_1^{-2} \right)\\
&\leq \tilde \varphi(x_1,t_1)\leq u(x_1,t_1)+ 16\a \tilde\Lambda^{-1} \rho_1^{-2}\\
&\leq \mathfrak{F}'(t_1)+ 16\a \tilde\Lambda^{-1} \rho_1^{-2}\\
\end{split}
\end{equation}

By choosing $\tilde \Lambda(n,\a,K)>1$ sufficiently large so that $\tilde\Lambda^2 e^{-4\a K}>16$, we see that $t_1> S_2$ for some $S_2(K,\tilde \Lambda,C_2,C_1)>0$ which is independent of $\e\to 0$. By choosing $\tilde S=\frac12\min\{S_2,C_2^{-1}C_1,1,\exp(-\Lambda C_2C_1^{-1})\}$, we see that the conclusion holds on $B_{g_0}(x_0,1)\times (0,S\wedge \tilde S]$. This completes the proof by evaluating at $x_0$.
\end{proof}
}

We end this section by pointing out that the lowest eigenvalues of $\mathrm{Rm}$ with respect to various curvature cones satisfy the evolution equation required in Proposition~\ref{prop:MP}.  
\begin{lma}\label{lma:curvSatisfyingProp}
Let $g(t)$ be a smooth solution to the Ricci flow. Suppose $\mathcal{C}$ is one of the following curvature cone:
\begin{enumerate}
\item[(i)] non-negative curvature operator;
\item[(ii)]2-non-negative curvature operator;
\item[(iii)] weakly $\mathrm{PIC}_2$;
\item[(iv)]weakly $\mathrm{PIC}_1$;
\item[(v)] non-negative bisectional curvature in case  $g(t)$ is \K;
\item[(vi)]non-negative orthogonal bisectional curvature $\mathrm{OB}$ in case  $g(t)$ is K\"ahler;
\item[(vii)] non-negative Ricci curvature if $n=3$,
\end{enumerate}
then the function 
$$\varphi(x,t)=\inf\{s> 0: \mathrm{Rm}(x,t)+s\cdot  \mathrm{Id}\in \mathcal{C}\}$$
is non-negative continuous and satisfies $\rheat \varphi\leq \mathcal{R}\varphi+c_n\varphi^2$ in the sense of barrier for some dimensional constant $c_n\geq 0$.
\end{lma}
\begin{proof}
When $n\geq 4$, the case of (i)-(v) follows from \cite{BamlerCabezasWilking2019} while (vi) is proved in \cite{LiNi2020}. It remains to prove (vii) which is equivalent to (iv) when $n=3$.  By \cite{Hamilton1982}, if we use $\varphi$ to denote the negative part of the Ricci lower bound when $n=3$, then whenever $\varphi\geq 0$,
\begin{equation}
\begin{split}
\rheat \varphi&=-(\mu-\nu)^2+\varphi(\mu+\nu)
\end{split}
\end{equation}
where $\mu\geq \nu\geq -\varphi$ denotes the eigenvalues of $\Ric(x,t)$. Since $\mu+\nu-\varphi=\mathcal{R}$, results follows with $c_3=1$.
\end{proof}

\section{Volume non-collapsing along flow}

In view of Proposition~\ref{Prop:CurvEsti}, the volume non-collapsing along the flow plays an important role to avoid curvature blowup. In this section, we will show that under the mildly integrable curvature lower bound, if the initial metric is volume non-collapsed, then so does the flow for a uniform short time.  We rely heavily the idea of Simon-Topping \cite{SimonTopping2023} which considered the case when the Ricci curvature is uniformly bounded from below along the flow.

We start with a distance distortion under a weaker condition. Namely, $\Ric(g(t))\geq -\phi$ for some $\phi\in L^1([0,S])\cap C^0_{loc}((0,S])$ along the flow. In practice, it will be sufficient to take $\phi(t)=t^{-1} (-\log t)^{1+\frac12}$.
\begin{lma}\label{lma:DistanceD}
Suppose $(M^n,g(t)),t\in [0,S]$ is a smooth solution to the Ricci flow (not necessarily complete) such that for some $x_0\in M$,  we have $B_{g(t)}(x_0,5)\Subset M$ for all $t\in [0,S]$.  Then the following holds.
\begin{enumerate}
\item[(a)] If the Ricci curvature satisfies $\Ric(g(x,t))\geq -\phi(t)$ for some non-negative function $\phi(t)\in L^1([0,S])\cap C^0_{loc}((0,S])$ and for all $x\in B_{g(t)}(x_0,5)$, $t\in (0,S]$, then there exists $\tilde S>0$ depending only on $\phi$ such that for any $x,y\in B_{g(0)}(x_0,1)$ and $0\leq s<t\leq S\wedge \tilde S$,  we have 
$$d_{g(t)}(x,y)\leq\exp\left(\int^t_s \phi(z)\,dz \right) \cdot d_{g(s)}(x,y).$$
\item[(b)] If $|\Rm(x,t)|\leq \a t^{-1}$ on $B_{g(t)}(x_0,5)$ for $t\in(0,S]$, then for all $0\leq s\leq t\leq S$, 
$$B_{g(s)}(x_0,1-\b_n\sqrt{\a s})\supset B_{g(t)}(x_0,1-\b_n\sqrt{\a t}).$$
\end{enumerate}
\end{lma}
\begin{proof}
We first prove (a) following the idea in \cite[Lemma 3.1]{SimonTopping2023}. Define 
$$S_1=\sup\{ s\in [0,S]: B_{g_0}(x_0,1)\subset B_{g(t)}(x_0,5)\}.$$
Hence, for $t\in (0,S_1]$, we have Ricci lower bound $\phi(t)$. In particular, for all $x\in B_{g_0}(x_0,1)$ and $t\in (0,S_1]$,
\begin{equation}
\begin{split}
d_{g(t)}(x,x_0)\leq \exp\left(\int^t_0 \phi(s)\,ds \right) \cdot d_{g_0}(x,x_0).
\end{split}
\end{equation}
Hence,  there exists $\tilde S(\phi)>0$ such that $S_1\geq S\wedge \tilde S$. The conclusion follows using the Ricci lower bound again. The property (b) follows directly from \cite[Corollary 3.3]{SimonTopping2023}. This completes the proof.
\end{proof}

The next result shows that the volume non-collapsing will be preserved in the  scale of $\sqrt{t}$ under the scaling invariant curvature  bound and $L^1$ (in time) Ricci lower bound.  
\begin{prop}\label{prop:volume-noncollapsed} 
Suppose $(M,g(t))$ is a Ricci flow for $t\in [0,S]$ (not necessarily complete) such that $g(0)=g_0$. If for some $x_0\in M$, $B_{g(t)}(x_0,10)\Subset M$ for all $t\in [0,S]$ and for some $v_0,\a>0$,  we have
\begin{enumerate}
\item[(i)] $v_0^{-1}r^n\geq \mathrm{Vol}_{g_0}(B_{g_0}(x,r))\geq v_0r^n$ for $r\in (0,1]$ and $x\in B_{g_0}(x_0,5)$;
\item[(ii)] $|\Rm(x,t)|\leq \a t^{-1}$ on $B_{g(t)}(x_0,10)$ for $t\in (0,S]$;
\item [(iii)] for all $x\in B_{g(t)}(x_0,10)$ and $t\in (0,S]$, the Ricci curvature satisfies
$$\Ric(g(x,t))\geq -\phi(t)$$ for some non-negative decreasing function $\phi(t)\in L^1([0,S])\cap C^0_{loc}((0,S])$.
\end{enumerate}
Then there exist $\tilde \e_0(n,v_0,\phi)$ and $\tilde S(n,v_0,\a,\phi)>0$ such that for all $t\in (0,\tilde S\wedge S]$ and $x\in B_{g(t)}(x_0,1)$,
$$\mathrm{Vol}_{g(t)}\left( B_{g(t)}(x,\sqrt{t}) \right) \geq \tilde \e_0 t^{n/2}.$$
\end{prop}
\begin{rem}
In the presence of uniform Ricci lower bound, the Ahlfors $n$-regularity, i.e. assumption (i), follows easily from volume lower bound of $B_{g_0}(x_0,10)$ and volume comparison theorem. 
\end{rem}
\begin{proof}
The proof is almost identical to that of \cite[Lemma 11.1]{SimonTopping2023}. We only give a sketch and point out the necessary modifications. We first note that since $\phi\in L^1([0,S])$, for any $\e>0$, there exists $S_\e$ such that if $0<t<S_\e$,
\begin{equation}
\begin{split}
\e\geq \int^{2t}_{t} \phi(s) ds \geq t \phi(t).
\end{split}
\end{equation}
Hence,  as $t\to0^+$,  the Ricci lower bound behaves as
\begin{equation}\label{asym-Ric-lowerBdd}
\Ric(g(t))\geq -\phi(t)\geq -o(1)t^{-1}.
\end{equation}
This serves as the crucial replacement of the uniform Ricci lower bound assumption in \cite[Lemma 11.1]{SimonTopping2023}. 

 With this replacement in mind, we can carry out the exact same argument in the proof of \cite[Lemma 11.1]{SimonTopping2023} to obtain the following. If the conclusion fails, then for any sufficiently small $\tilde \e_0$, there exist a sequence of $n$-manifold $M_i$, a sequence of point $x_i\in M_i$, a sequence of Ricci flows $\tilde g_i(t)$ on $M_i\times [0,\tilde t_i]$ with $\tilde t_i$ decreases to $0$, $t_i\in (0,\tilde t_i]$ and a sequence of points $ y_i\in B_{\tilde g_i(t_i)}(x_i,1)$ such that 
\begin{enumerate}
\item[(i)] $B_{\tilde g_i(t)}(y_i,7)\Subset M_i$ for all $t\in [0,t_i]$;
\item[(ii)] $\Ric(\tilde g_i(t))\geq -\phi(t)$ on $B_{\tilde g_i(t)}(y_i,7)$ for all $t\in [0,t_i]$;
\item[(iii)] $|\Rm(\tilde g_i(t))|\leq \a t^{-1}$ on $B_{\tilde g_i(t)}(y_i,7)$ for all $t\in [0,t_i]$;
\item[(iv)] $v_0^{-1}r^n\geq \mathrm{Vol}_{\tilde g_i(0)}(y,r)\geq v_0r^n$ for all $r\in (0,1)$ and $y\in B_{\tilde g_i(0)}(y_i,1)$;
\item[(v)]$\mathrm{Vol}_{\tilde g_i(t_i)}(y_i,\sqrt{t_i})= \tilde \e_0 (\sqrt{t_i})^n$;
\item[(vi)] $\mathrm{Vol}_{\tilde g_i(t)}(y,\sqrt{t})\geq  2^{-n-1}\tilde \e_0 (\sqrt{t})^n$ for $t\in [0,t_i]$ and all $y\in B_{\tilde g_i(t)}\left(y_i,c_n\a^{-1}\sqrt{\frac{t_i}{\tilde t_i}}\right)$;
\end{enumerate}
We remark here that  since $t_i\to 0$, \cite[(11.15)]{SimonTopping2023} now follows from \eqref{asym-Ric-lowerBdd} when $i$ is sufficiently large. Now consider the rescaled Ricci flow $g_i(t)=t_i^{-1}\tilde g_i(t_it),t\in [0,1]$ so that the smooth sub-sequential Cheeger-Gromov limit $(M_\infty,g_\infty(t),y_\infty)$ of $(M_i,g_i(t),y_i), t\in (0,1]$ satisfies
\begin{enumerate}
\item[(I)] $\Ric(g_\infty(t))\geq 0$ for all $t\in (0,1]$;
\item[(II)] $|\Rm(g_\infty(t))|\leq \a t^{-1}$ for all $t\in (0,1]$;
\item[(III)] $\mathrm{Vol}_{g_\infty(1)}\left(B_{g_\infty(1)}(y_\infty,1)\right)= \tilde \e_0$;
\item[(IV)] $\mathrm{Vol}_{g_\infty(t)}\left( B_{g_\infty(t)}(y,\sqrt{t})\right)\geq 2^{-n-1} \tilde \e_0 t^{n/2}$ for all $t\in (0,1]$ and $y\in M_\infty$;
\item[(V)] For any $\hat R>2r>0$ and $z_i\in B_{g_i(0)}(y_i,\hat R)$, 
$$\limsup_{i\to +\infty}\mathrm{Vol}_{\tilde g_i(0)}\left( B_{\tilde g_i(0)}(z_i,r) \right)\leq v_0^{-1}r^n.$$
\end{enumerate}
Here the property (I) follows from \eqref{asym-Ric-lowerBdd} while (V) follows from the initial Ahlfors $n$-regularity assumption. Moreover, using the same proof where \cite[Lemma 3.4]{SimonTopping2023} is now replaced by Lemma~\ref{lma:DistanceD}, the {\bf Claim 1} in \cite[page 513]{SimonTopping2023} also holds for $g_\infty(1)$ in the sense that if we have a ball $B_{g_\infty(1)}(y,L)$ for $y\in M_\infty$ and $L>0$, that is covered by $N$ balls of radius $r\geq R(n,\a)$, then we must have $N\geq \eta(n)\cdot v_0^2 L^n r^{-n}$. Now we are in position to apply \cite[Lemma 11.2]{SimonTopping2023} as in the proof of \cite[Lemma 11.1]{SimonTopping2023} to derive contradiction when $\tilde \e_0$ is chosen to be too small.
\end{proof}

Finally, we show that the volume non-collapsing is indeed preserved in a large scale for a uniform short time.
\begin{prop}\label{prop:volume-noncollapsed-2}
Under the assumption of Proposition~\ref{prop:volume-noncollapsed}, there exist $\tilde v(n,v_0,\phi)$ and $\tilde S(n,v_0,\a,\phi)>0$ such that for all $t\in (0,S\wedge \tilde  S]$ and $x\in B_{g(t)}(x_0,1)$,
$$\mathrm{Vol}_{g(t)}\left( B_{g(t)}(x_0,1) \right) \geq  \tilde v.$$
\end{prop}

\begin{proof}

Since $B_{g(t)}(x_0,1)$ is compactly contained inside $M$, we might take finitely many $\{B_{g(t)}(x_i,\b_n\sqrt{\a t})\}_{i=1}^N$ such that $B_{g(t)}(x_i,\b_n\sqrt{\a t})\cap B_{g(t)}(x_j,\b_n\sqrt{\a t})=\emptyset$ for $i\neq j$ and $B_{g(t)}(x_i,\b_n\sqrt{\a t})\subseteq B_{g(t)}(x_0,1)$ while $\{   B_{g(t)}(x_i,4\b_n\sqrt{\a t})\}_{i=1}^N$ is a cover of $B_{g(t)}(x_0,1)$. Therefore,  Proposition~\ref{prop:volume-noncollapsed} implies 
\begin{equation}\label{vol-estim}
\begin{split}
\mathrm{Vol}_{g(t)}\left( B_{g(t)}(x_0,1) \right)&\geq \sum_{i=1}^N \mathrm{Vol}_{g(t)}\left( B_{g(t)}(x_i,\b_n\sqrt{\a  t}) \right)\geq N \gamma(\b_n^2\a  t)^{n/2}.
\end{split}
\end{equation}

It remains to estimate $N$ from below. By Lemma~\ref{lma:DistanceD}, we might restrict $S$ so that
\begin{equation}
\begin{split}
B_{g_0}\left(x_0,\frac12\right) \subset B_{g(t)}(x_0,1)\subset \bigcup_{i=1}^NB_{g(t)}(x_i,4\b_n\sqrt{\a  t})\subset  \bigcup_{i=1}^N B_{g_0}(x_i,5\b_n\sqrt{\a  t}).
\end{split}
\end{equation}
By measuring the set using $d\mathrm{vol}_{g_0}$ and volume growth assumption of $g_0$, we deduce that 
\begin{equation}
 C_0^{-2} v_0 \leq 5^n N \left(\b_n\sqrt{\a  t}\right)^n
\end{equation}
This completes the proof by substituting it back to \eqref{vol-estim}.
\end{proof}

\section{Ricci flow smoothing on Compact manifolds}

In this section, we establish a existence theory of Ricci flow and use it to prove the stability. 
\begin{proof}[Proof of Theorem~\ref{thm:RF-compact}]
By the short-time existence Theorem of Hamilton \cite{Hamilton1982}, there exists a maximal solution $g(t),t\in [0,T_{max})$ to the Ricci flow on $M$.  Let $\a$ to be a large constant to be determined. For the given $\a$, we let $s_0$ be the superemum of $s\in [0,T_{max})$ such that for all $(x,t)\in M\times [0,s)$,
\begin{enumerate}
\item[(i)] $|\Rm(x,t)|< \a t^{-1}$;
\item [(ii)] $\mathrm{inj}(g(x,t))> \sqrt{\a^{-1}t}$.
\end{enumerate}

By the smoothness of flow,  we have $s_0\in (0,T_{max})$.  Hence, the condition (i)-(ii) hold on $[0,s_0)$ and there exists $x_0\in M$ such that either 
\begin{enumerate}
\item[(I)]$|\Rm(x_0,s_0)|= \a s_0^{-1}$; or
\item[(II)]$\mathrm{inj}(g(x_0,s_0))= \sqrt{\a^{-1}s_0}$.
\end{enumerate}

By Lemma~\ref{lma:curvSatisfyingProp} and Lemma~\ref{lma:SKT-Morry}, $\varphi=(\mathrm{IC}_1)_-$ satisfies the assumptions in Proposition~\ref{prop:MP} with $r=\sqrt{T}$.  Hence, it applies to see that there exist $S_1(n,\a,f,T,\Lambda),\tilde\Lambda(n,\a,f,T,\Lambda)>0$ so that for all $(x,t)\in M\times (0,S_1\wedge s_0]$,
$$\varphi(x,t)\leq \tilde\Lambda \left[ \frac{f(\tilde\Lambda t)}{t}+1\right]\leq \tilde\Lambda \left[ 1+ \frac1{t}\left(\frac{\Lambda}{\log (T/\tilde\Lambda)+\log t^{-1}} \right)^2\right]$$
where we have used the fact that $f$ satisfies \eqref{SK-condition}:
\begin{equation}
\begin{split}
\Lambda \geq \int^T_t \frac{\sqrt{f(s)}}{s}ds\geq \sqrt{f(t)}\cdot \log \left(\frac{T}{t} \right).
\end{split}
\end{equation}

By taking $\phi(t)=t^{-1} (-\log t)^{1+\frac12}\in L^1([0,T])\cap C^0_{loc} ((0,T])$, we see that if we shrink $S_1$ further (depending also on $\tilde\Lambda$), then for all $(x,t)\in M\times (0,S_1\wedge s_0]$,
\begin{equation}\label{Proof"IC1}
\begin{split}
\varphi(x_0,t)\leq  \tilde\Lambda \left[ 1+ \frac1{t}\left(\frac{\Lambda}{\log (T/\tilde\Lambda)+\log t^{-1}} \right)^2\right] \leq \phi(t).
\end{split}
\end{equation}

Together with Proposition~\ref{prop:vol-SK} and assumption on the volume non-collapsing, we might apply Proposition~\ref{prop:volume-noncollapsed} to conclude that there exists $\gamma(n,f,T,\Lambda,v)>0$ such that for all $(x,t)\in M\times (0,S_1\wedge s_0]$,
\begin{equation}\label{Proof"non-collap}
\mathrm{Vol}_{g(t)}\left( B_{g(t)}(x,\sqrt{t}) \right) \geq  \gamma t^{n/2}.
\end{equation}
Moreover,  the bound of $\varphi$ at the same time implies $\Ric(x,t)\geq -(n-1)\phi(t)\geq -t^{-1}$ as $t\to 0^+$ and hence applying the volume comparison to $t^{-1}g(t)$ implies that for all $(x,t)\in M\times (0,S_1\wedge s_0]$ and $r\leq \frac14\sqrt{t}$,
$$\mathrm{Vol}_{g(t)}\left( B_{g(t)}(x,r) \right) \geq  \tilde v r^n$$
for some $\tilde v(n,f,v,D,T,\Lambda)>0$ independent of $\a$.  

Now we can apply Proposition~\ref{Prop:CurvEsti} to show that for all $(x,t)\in M\times (0,S_1\wedge  s_0]$, 
\begin{equation}\label{improved-estimate}
|\Rm(x,t)|\leq \b_0 t^{-1}
\end{equation}
for some $\b_0(n,f,v,D,T,\Lambda)>0$.  By combining \eqref{improved-estimate},  \eqref{Proof"non-collap} and injectivity radius estimate of Cheeger-Gromov-Taylor \cite{CheegerGromovTaylor1982}, we conclude that for all $(x,t)\in M\times (0,S_1\wedge s_0]$, $\mathrm{inj}(x,t)\geq \sqrt{\b_1^{-1}t}$ for some $\b_1(n,f,v,D,T,\Lambda,\b_0)>0$. Therefore,  if we fix $\a=\b_1+\b_0$, then we see that $s_0>S_1$. Hence, the conclusion holds on $(0,S_1]$ for this choice of $\a$ and the corresponding $S_1>0$.  This proves (a) and (c) while (d) follows from \eqref{Proof"IC1}. Finally (b) follows from Proposition~\ref{prop:volume-noncollapsed-2} and covering argument after possibly shrinking $S_1$ further. This completes the proof.
\end{proof}

In higher dimension, when $g_0$ is in addition K\"ahler, one usually consider the holomorphic curvature which is more complex geometric.  We say that the holomorphic bisectional curvature $\mathrm{BK}$ of a \K curvature tensor $\mathrm{R}$ is non-negative  (denoted by $R\in \mathrm{C}_{\mathrm{BK}}$) if for any $X,Y\in T^{1,0}M$, we have 
$$R(X,\bar X,Y,\bar Y)\geq 0.$$  If one instead ask for the non-negativity only for $X,Y\in T^{1,0}M$ such that $g(X,\bar Y)=0$, then we say that the orthogonal bisectional curvature $\mathrm{OB}$ is non-negative (denoted by $R\in \mathrm{C}_{\mathrm{OB}}$).  One might then analogously define the class of compact \K manifolds satisfying the strong Kato-type lower bound on $\mathrm{BK}$ (or $\mathrm{OB}$). Define for $t>0$, 
$$\kappa_{t,\mathrm{BK}}(M^n,g_0)=\sup_{x\in M} \int^t_0 \int_M H_{g_0}(x,y,s)(\mathrm{BK})_-(y)\, d\mathrm{vol}_{g_0}(y) ds$$
where $\mathrm{BK}_-(x)=\inf\{ \a\geq 0: \mathrm{Rm}_{g(x)}+\a B\in \mathrm{C}_{\mathrm{BK}} \}$ and $B$ is the \K curvature tensor given by $B(X,\bar Y,Z,\bar W)=g(X,\bar Y)g(Z,\bar W)+g(X,\bar W)g(Z,\bar Y)$.  In this case, the same proof of Theorem~\ref{thm:RF-compact} yields the following.
\begin{thm}\label{thm:kahler-case}
Suppose $(M^n,g_0,x_0)$ is a pointed compact \K manifold with $\mathrm{dim}_\C(M)=n\geq 2$.  Suppose $(M^n,g_0,x_0)\in  \mathcal{K}_{\mathrm{OB}}(n,f,v)\cap  \mathcal{K}(n,f,v)$ for some non-decreasing function $f$, $T$ and $\Lambda$ satisfying \eqref{intro:SK-condition} and $\mathrm{diam}(M,g_0)\leq D$ for some $D>0$, then there exist $S,\a,\tilde v$ depending only on $n,f,v,D,T,\Lambda$ and a solution $g(t),t\in [0,S]$ to the \KR  flow such that for all $(x,t)\in M\times (0,S]$,
\begin{enumerate}
\item[(a)] $|\Rm(x,t)|\leq \a t^{-1}$;
\item [(b)] $\mathrm{Vol}_{g(t)}\left(B_{g(t)}(x, \sqrt{T})\right)\geq \tilde v T^{n/2}$;
\item[(c)]$\mathrm{inj}(g(x,t))\geq \sqrt{\a^{-1}t}$;
\item[(d)] $\mathrm{Rm}(g(x,t))+\a (1+t^{-1}f(\a t))\cdot \mathrm{B} \in \mathrm{C}_{\mathrm{BK}}$.
\end{enumerate}
\end{thm}

When a sequence of compact manifold $(M_i^n,g_i)$ satisfy a uniform Dynkin-type lower bound of Ricci, it is already shown by Carron \cite{Carron2019} that the sequence admits a Gromov-Hausdorff limit after passing to subsequence.  The structure theory was further investigated by Carron-Mondello-Tewodrose \cite{CarronMondelloTewodrose2021,
CarronMondelloTewodrose2022} in the non-collapsed case when the Dynkin-type lower bound of Ricci is strengthened to be a strong Kato-type lower bound.  We now shows that in dimension three,  the limit will be Gromov-Hausdorff close and homeomorphic to a smooth manifold. Moreover, the result also holds in higher dimension if we require the strong Kato bound on $1$-isotropic curvature.

\begin{proof}[Proof of Corollary~\ref{cor:RF-fromSing}]
The proof is almost identical to that of \cite[Corollary 4]{BamlerCabezasWilking2019}. We include it for reader's convenience.
By Theorem~\ref{thm:RF-compact},  there exist $\a,S>0$ and $\phi\in L^1([0,S])\cap C^0_{loc}((0,S))$ such that for each $i\in \mathbb{N}$, we can find a Ricci flow $g_i(t),t\in [0,S]$ such that $g_i(0)=g_{i,0}$ and for all $t\in (0,S]$,
\begin{enumerate}
\item[(a)] $|\Rm(g_i(t))|\leq \a t^{-1}$;
\item[(b)]$\mathrm{inj}(g_i(t))\geq \sqrt{\a^{-1}t}$;
\item[(c)] $\mathrm{IC}_1(g_i(t))\geq -\phi(t)$.
\end{enumerate}

It follows from Hamilton's compactness \cite{Hamilton1995} that after passing to subsequence, there exists a smooth Ricci flow $(M_\infty,g_\infty(t),x_\infty),t\in (0,S]$ such that $(M_i,g_i(t),x_i)$ converges to $(M_\infty,g_\infty(t),x_\infty)$ in the $C^\infty$ Cheeger-Gromov sense.  By Lemma~\ref{lma:DistanceD},  for all $i$ and $0<s\leq t\leq S$
$$\exp\left(-\int^t_s \phi(z)\,dz \right) \cdot d_{g_i(t)}\leq d_{g_i(s)}\leq d_{g_i(t)}+C(\sqrt{t}-\sqrt{s}).$$

If we first let $i\to +\infty$ and followed by $s\to 0$, we see that $M_\infty$ is compact since $\mathrm{diam}(M_i,g_{i,0})\leq D$ and $d_{g_\infty(s)}\to d_\infty$ for some distance function $d_\infty$ on $M_\infty$.  By interchanging the order of taking limit,  it is easy to see that $d_{g_{i,0}}$ converges to $d_\infty$ modulus diffeomorphism. This proved the Gromov-Hausdorff convergence.  It follows from \cite[Theorem A]{CarronMondelloTewodrose2021} that the convergence is also in the measured Gromov-Hausdorff sense.

Suppose now $(M_i,g_{i,0},x_i)\in \mathcal{K}_{\mathrm{IC1}}(n,f_i,v)$ so that $f_i(T)$ decreases to $0$.  We want to show that $g_\infty(t)$ satisfies a stronger curvature estimate. 

Fix any $x_i\in M$. If $r\geq D$ where $\mathrm{diam}(M_i,g_{i,0})\leq D$, then for all $x\in B_{g_{i,0}}(x_i,r)=M_i$ and $t\in (0,r^2\wedge T]$, 
\begin{equation}
\begin{split}
&\quad \int_{B_{g_{i,0}}(x_i,2r)} \frac1{t^{n/2}} \exp\left(-\frac{C_1}{t} d^2_{g_{i,0}}(x,y) \right)  \varphi_{i,0}(y) \, d\mathrm{vol}_{g_{i,0}}(y) \\
&=\int_{M_i} \frac1{t^{n/2}} \exp\left(-\frac{C_1}{t} d^2_{g_{i,0}}(x,y) \right)  \varphi_{i,0}(y) \, d\mathrm{vol}_{g_{i,0}}(y) \leq \frac{C_1f_i(t)}t
\end{split}
\end{equation}
by Lemma~\ref{lma:SKT-Morry} and assumption.  While if $t\in [T,r^2]$, then the diameter bound yields
\begin{equation}
\begin{split}
&\quad \int_{B_{g_{i,0}}(x_i,2r)} \frac1{t^{n/2}} \exp\left(-\frac{C_1}{t} d^2_{g_{i,0}}(x,y) \right)  \varphi_{i,0}(y) \, d\mathrm{vol}_{g_{i,0}}(y) \\
&\leq e^{C_1D^2(T^{-1}-t^{-1})}\left(\frac{T}{t} \right)^{n/2}\int_{M_i}\frac1{T^{n/2}} \exp\left(-\frac{C_1}{T} d^2_{g_{i,0}}(x,y) \right)  \varphi_{i,0}(y) \, d\mathrm{vol}_{g_{i,0}} (y)\\
&\leq C_1 e^{C_1D^2(T^{-1}-t^{-1})} \frac{f_i(T)}T.
\end{split}
\end{equation}

For $t\in [T,r^2]$, we extend $f_i(t)$ by defining 
$$f_i(t)=t\cdot e^{C_1D^2(T^{-1}-t^{-1})} \frac{f_i(T)}T.$$
Since $f_i(T)\to 0$ as $i\to +\infty$,  $f_i$ satisfies the assumption (iii) in Proposition~\ref{prop:MP} for $r\to +\infty$ in the sense that for any $r\geq D$,  there exists $N\in \mathbb{N}$ so that for all $i>N$,  Proposition~\ref{prop:MP} applies on $B_{g_0}(x_i,r)$ so that 
$$(\mathrm{IC}_1)_-(g_i(x_i,t))\leq \tilde\Lambda\left(\frac{f_i(\tilde \Lambda t)}{t}+\frac1{r^2}\right)$$
for some uniform $\tilde \Lambda>0$.  Since $x_i$ is arbitrary in $M_i$, by letting $i\to +\infty$ for $t\in (0,S]$ and followed by letting $r\to +\infty$, we see that $\mathrm{Rm}(g_\infty(t))\in \mathrm{C}_{\mathrm{PIC1}}$ on $M_\infty$ for $t\in (0,S]$.  This completes the proof.
\end{proof}

As another application, we obtain a stability in view of Geroch conjecture. This is motivated by the work of \cite{AllenBrydenKazaras2022} where compactness under uniform $L^p$ curvature bound for $p=3$ is studied when $n=3$.  Since a $L^p,p>\frac32$ lower bound on $\Ric$ implies the uniform strong Kato bound of $\Ric$ \cite{RoseStollmann2017}, together with the regularity theory of Ricci flow (for instances, see \cite{Yang1992-1,Yang1992-2}), we can view Corollary~\ref{cor:LimitinPSC} as a generalization. 
\begin{proof}[Proof of Corollary~\ref{cor:LimitinPSC}]
By Corollary~\ref{cor:RF-fromSing} it suffices to identify the Gromov-Hausdorff limit. Recall that there exists $(M_\infty,g_\infty(t),x_\infty),t\in (0,S]$ such that the Ricci flow $(M_i^n,g_{i}(t),x_i)$ with $g_i(0)=g_{i,0}$  converges to $(M_\infty,g_\infty(t),x_\infty)$ in the smooth Cheeger-Gromov sense after passing to subsequent.  By maximum principle, we have $\mathcal{R}(g_i(t))\geq -i^{-1}$ for all $i$ and hence $\mathcal{R}(g_\infty(t))\geq 0$.  Since $M_\infty$ is compact, the smooth convergence implies that $M_\infty$ is diffeomorphic to $M_i$ for $i$ sufficiently large. In particular, $\sigma(M_\infty)\leq 0$ and hence $g_\infty(t)$ is Ricci flat for $t\in (0,S]$ by strong maximum principle.  Thus, $g_\infty=g_\infty(S)$ is desired limit. 

If $b_1(M_i)=n$ for all $i$,  we also have $b_1(M_\infty)=b_1(M_i)=n$ for $i$ large. Since $g_\infty(t)$ is of Ricci flat,  it is the standard torus by the classical Bochner theorem. This completes the proof.
\end{proof}

\end{document}